\documentclass[11pt,reqno]{amsart}

\usepackage{amsmath, mathtools}
\usepackage{amsthm}
\usepackage{graphicx}
\usepackage{upgreek}
\usepackage{hyperref}
\usepackage{mathrsfs}
\usepackage{xcolor}
\usepackage{bm,bbm}
\usepackage{amsfonts}
\usepackage{amssymb}
\usepackage{csquotes}
\usepackage{stmaryrd}
\usepackage{ulem}
\usepackage{cancel, soul}
\usepackage{enumerate}

\usepackage{soul}
\usepackage[margin=1.1in]{geometry}
\numberwithin{equation}{section}

\newtheorem{thm}{Theorem}
\newtheorem{lem}[thm]{Lemma}
\newtheorem{prop}[thm]{Proposition}

\newtheorem{rem}[thm]{Remark}

\newcommand{\noi}{\noindent}

\newcommand{\E}{\mathbb{E}}
\newcommand{\R}{\mathbb{R}}

\newcommand{\N}{\mathbb{N}}

\newcommand{\ph}{\varphi}

\newcommand{\Q}{{\mathbb Q}}

\newcommand{\PP}{{\mathbb P}}

\newcommand{\calA}{{\mathcal A}}

\newcommand{\calC}{{\mathcal C}}

\newcommand{\calF}{{\mathcal F}}

\newcommand{\calL}{{\mathcal L}}

\newcommand{\calQ}{{\mathcal Q}}

\newcommand{\calV}{{\mathcal V}}

\newcommand{\To}{\Rightarrow}

\newcommand\iy{\infty}

\newcommand{\pen}{ g}

\newcommand{\one}{\mathbbm{1}}
\allowdisplaybreaks

\newcommand{\sredm}[1]{\ifmmode\text{\xout{\ensuremath{\displaystyle \textcolor{red}{#1}}}}\else\sout{\textcolor{red}{#1}}\fi}
\title{A Scaling Limit for Utility Indifference Prices in the Discretized Bachelier Model}
\author[A. Cohen]{Asaf Cohen }
\address{Department of Mathematics\\
University of Michigan\\
Ann Arbor, MI 48109\\
United States
}
\email{shloshim@gmail.com }

\author[Y. Dolinsky]{Yan Dolinsky }
\address{Department of Statistics\\
Hebrew University\\
Jerusalem\\
Israel}
\email{yan.dolinsky@mail.huji.ac.il}

\thanks{This is the last version of the paper. To appear in {\it Finance and Stochastics}.}
\thanks{A. Cohen acknowledges the financial support of the National Science Foundation (DMS-2006305).
Y. Dolinsky is supported in part by the GIF Grant 1489-304.6/2019 and the ISF grant 160/17}
\date{\today}

\begin{document}

\maketitle

\begin{abstract}
We consider the discretized Bachelier model
where hedging is done on an equidistant set of times.
Exponential utility indifference prices are studied for path-dependent European options and we compute their non-trivial scaling limit for a
large number of trading times $n$ and when
 risk aversion is scaled like $n\ell$ for some constant $\ell>0$.
 Our analysis is purely probabilistic. We first use a duality argument to transform the problem into an optimal drift control problem with a penalty term. We further use
martingale techniques and strong invariance principles and get that the limiting problem takes the form of a volatility control problem.

${}$\\
\noi{\bf AMS Classification:} Primary: 91G10, 60F15

\noi{\bf Keywords:} Utility indifference, strong approximations, path-dependent SDEs, asymptotic analysis
\end{abstract}
\section{Introduction}

\vspace{3pt}
Taking into account market frictions is an important challenge in financial modeling.
In this paper, we focus on the friction that the
rebalancing of the portfolio strategy is limited to occur discretely.
In such a realistic
situation, a general future payoff cannot be hedged perfectly even in complete market models
such as the Bachelier model or the Black--Scholes model.

\vspace{3pt}
We consider hedging of a path-dependent European contingent claim in the Bachelier model
for the setup where the investor can hedge on
an equidistant set of times.
Our main result is providing the asymptotic behavior of the exponential utility indifference prices
where the risk aversion goes to infinity linearly
in the number of trading times.
Namely, we establish a non trivial scaling limit for
indifference prices where the friction goes to zero and the risk aversion goes to infinity.

\vspace{3pt}
This type of scaling limits goes back to the
seminal work of Barles and Soner \cite{BS:98}, which determines
the scaling limit of
utility indifference prices of vanilla
options for small proportional
transaction costs and high risk aversion.
Another work in this direction is the recent article \cite{BD:20} which deals with
scaling limits of
utility indifference prices of vanilla options for hedging with vanishing delay $H\downarrow 0$ when
risk aversion is scaled like $A/H$ for some constant $A$. In general, the common ground between the above two works and the present paper
it that all of them
start with complete markets and consider small frictions, which make the markets incomplete and the derivative securities cannot be perfectly hedged with a reasonable initial capital. Then, instead of considering perfect hedging,
these papers study utility indifference prices with exponential utilities and with large risk aversion. In contrast to the previous two papers which treated only vanilla options (path-independent), in this paper we are able to provide the limit theorem for path-dependent options.

\vspace{3pt}
Although the topic of discrete time hedging in the Brownian setting was largely studied,
the corresponding papers
studied the
 optimal discretization of given hedging strategies or stochastic integrals.
Indeed, the papers \cite{BKL:00,GT:01,HM:05} studied the convergence rate of a discrete time delta hedging strategies
for the case where the trading is done
on equidistant set of times.
In
\cite{G:02} the author proposed to discretize delta hedging strategies with
non-equidistant deterministic time nets, and showed that this generalizations leads (for some payoffs)
to a better error estimates. For further research in this direction see \cite{G:05,GT:09,GM:12}.
The papers \cite{F:11,F:14,CFRT:16}
study approximation of stochastic integrals with a discretization procedure that goes beyond deterministic nets and is performed on a set of random (stopping) times.

In the present study, instead of tracking a given hedging strategy
we follow the well known approach of utility indifference pricing which is commonly used in
the setup of incomplete markets (see \cite{R:08} and the references therein).
In other words, this approach says that the
price of a given contingent claim should be equal to the minimal amount of money that an investor has to
be offered so that she becomes indifferent (in terms of utility) between the situation where she has sold the
claim and the one where she has not.

\vspace{3pt}
We now put our contribution in the context of asymptotic analysis of risk-sensitive control problems. Such problems model situations where a decision maker aims to minimize small probability events with significant impacts.
Typically, the small probability event emerges from a state process with a volatility that vanishes with the scaling parameter. The limiting behaviors of such risk-sensitive control problems are governed by deterministic differential games, see e.g., \cite{Fleming2006b} and the references therein. In our case, the volatility of the state process does not vanish, but is rather of order $O(1)$, and the small probability event emerges from the discrete approximation of the stochastic integral on the grid. As a result, the structure of the limiting problem is quite different: rather than a deterministic (drift control) differential game, we obtain a stochastic (volatility) control problem. The connection between the indifference price and the difference between two values of two (non asymptotically) risk-sensitive problems is well-known, see e.g., \cite{hernandez2008relation}. It stems from the fact that the utility indifference pricing is a normalized version of the certainty equivalent criterion. Our study is concerned with the scaling limit of the utility indifference prices when the market friction goes to zero and the risk-aversion goes to infinity.

\vspace{3pt}
Let us outline the key steps in establishing the asymptotic result.
Our approach is purely probabilistic and allows to consider European contingent claims with path-dependent payoffs.
The first step in establishing the main result goes through a dual representation of the value function (Proposition \ref{prop.dual}). This representation is closer in nature to the form of the
limiting stochastic volatility control problem. In the dual problem, there is only one player: a maximizing (adverse) player that controls the drift;  the investor's role is translated into a martingale condition. The control's cost is small, hence allowing the maximizer to choose controls with high values. The second and the main step is to analyze the limit behavior of the dual representation. This is the main technical challenge of the paper and is given in Theorem \ref{thm.2}. The proof is done via upper and lower bounds.

\vspace{3pt}
There are two main challenges in the proof of the upper bound. The first one is comparing between the dual and the limiting penalty terms. The second one is due to the fact that the consistent price systems in the dual representation are not necessary tight. To handle the first challenge, we work on the discrete-time grid,
and for any level of penalty in the limiting problem we are able to construct an optimal penalty term in the prelimit problem (Lemma \ref{lem:main1}). The structure of this best penalty term also serves us in the proof of the lower bound.
We overcome the second difficulty by applying a strong invariance principle (Lemma \ref{lem:main2}) and not the widely used weak convergence approach which is not helpful here. Specifically, for any control in the prelimit dual problem we construct a process in the form of the limiting problem, which is close in probability to the original process and for which the penalization term is bounded in the desired direction.

\vspace{3pt}
The proof of the lower bound is made by an explicit construction driven by Lemma \ref{lem:main1}.
A key ingredient in the proof is a construction of a path-dependent stochastic differential equation (SDE),
which translates the consistent price systems given via the drift control into the limiting volatility control problem.

\vspace{3pt}
The rest of the paper is organized as follows. In Section \ref{sec:2} we introduce the model and formulate the main result (Theorem \ref{thm.1}). In Section \ref{sec:3} we provide the duality representation (Proposition \ref{prop.dual}), we formulate the main technical statement (Theorem \ref{thm.2}), and use these two results to deduce Theorem 1. The proof of Theorem \ref{thm.2} is given in Section \ref{sec:4}.

%\vspace{3pt}
%In the rest of this section we will list some frequently used notation.
%\subsection{Notation}
%$\R$ is the set of real numbers and $\N$ is the set of all (strictly) positive integers. For any $n\in\N$, the notation $[n]$ stands for $\{0,1,\ldots, n\}$. For any two real numbers, we denote $a\wedge b=\min\{a,b\}$ and $a\vee b=\max\{a,b\}$. For any set $A$, $\one_A$ equals 1 on the event $A$ and 0 otherwise. For a random vector $X$ in a probability space, $\calL(X)$ stands for the probability law of $X$. For any interval $I\subset\R$, $\calC(I)$ denotes the space of real-valued and continuous functions defined on $I$. We endow this space with the sup-norm.
%
%

\section{The model and the main results}\label{sec:2}
Let
$(\Omega, \mathcal{F}, \mathbb P)$ be  a complete probability space carrying a
one-dimensional Wiener process
$(W_t)_{t \in [0,T]}$ with natural augmented filtration $(\mathcal{F}^W_t)_{t \in [0,T]}$ and time horizon $T \in (0,\infty)$.
We consider a simple financial market
with a riskless savings account bearing zero interest (for simplicity) and with a risky
asset $X$ with Bachelier price dynamics
\begin{equation}\label{2.bac}
X_t=X_0+\sigma W_t+\mu t, \qquad t \in [0,T],
\end{equation}
where $X_0>0$ is the initial asset price, $\sigma>0$ is the constant volatility and $\mu\in\mathbb{R}$ is the constant drift. These parameters are fixed throughout the paper.

Fix $n\in\mathbb N$ and consider an investor who can trade the risky asset
only at times from the grid $\left\{0,T/n,2T/n,...,T\right\}$.
For technical reasons,
in addition to the risky asset $(X_t)_{t \in [0,T]}$
we assume that the financial market contains short time horizon options with payoffs
in the spirit of power options. Formally,
for any $k=0,1,...,n$,
at time $kT/n$ the investor can buy but not sell European options that can be exercised at the time $(k+1)T/n$ with the payoff
$|X_{(k+1)T/n}-X_{kT/n}|^3$.
Denote by $h(n)$ the price of the above option. For simplicity, we will assume that the price does not depend on $k$. Moreover, we
will assume the following scaling:
\begin{equation}\label{2.scale1}
\lim_{n\rightarrow\infty} n^{3/2} h(n)=\infty
\end{equation}
and
\begin{equation}\label{2.scale2}
\lim_{n\rightarrow\infty} n h(n)=0.
\end{equation}
This investment
opportunity can be viewed
as an insurance against high values of the stock fluctuations.
Roughly speaking, the term $|X_{(k+1)T/n}-X_{kT/n}|^3$ is of order $O(n^{-3/2})$.
However since the payoff of this option is quite extreme we expect that the corresponding price will be more expensive
than $O(n^{-3/2})$. The scaling which is given by
(\ref{2.scale1})--(\ref{2.scale2}) says that the option price is more expansive than $O(n^{-3/2})$ but cheaper than $O(1/n)$.
\begin{rem}
It is possible to replace the payoff
$|X_{(k+1)T/n}-X_{kT/n}|^{3}$ with the payoff
$|X_{(k+1)T/n}-X_{kT/n}|^{2+\epsilon}$ for $\epsilon>0$
and assume that the new option price $\hat h(n)$ satisfies
$$
\lim_{n\rightarrow\infty} n \hat h(n)=0 \qquad \mbox{and} \qquad \lim_{n\rightarrow\infty} n^{1+\epsilon/2} \hat h(n)=\infty.
$$
However, for simplicity we work with power$=3$.
\end{rem}
In line with the above, the set $\calA^n$ of trading strategies for the $n$-step model consists of pairs
$(\gamma,\delta)=\{(\gamma_k,\delta_k)\}_{0\leq k\leq n-1}$
such that for any $k$, $\gamma_k,\delta_k$ are $\mathcal F^W_{kT/n}$-measurable
and in addition $\delta_k\geq 0$ (there is no short selling in the power options).
The corresponding portfolio value at the grid times is given by
\begin{align}
\notag
V^{\gamma,\delta}_{\frac{kT}{n}}&:=\sum_{i=0}^{k-1}\gamma_i \left(X_{(i+1)T/n}-X_{iT/n}\right)\\
\label{portfolio}
&\quad+\sum_{i=0}^{k-1}\delta_i \left(|X_{(i+1)T/n}-X_{iT/n}|^3-h(n)\right), \ \ k=0,1,...,n.
\end{align}

Next, let $\calC[0,T]$ be the space of all continuous functions $z:[0,T]\rightarrow\mathbb R$ equipped with the uniform norm
$||z||:=\sup_{0\leq t\leq T}|z_t|$ and let the function $f:\calC[0,T]\rightarrow [0,\infty)$ be Lipschitz continuous with respect to the uniform norm.
For the $n$-step model we consider a European option with payoff of the form
$f^n(X):=f(p^n(X))$
where $p^n(z)$ returns the linear interpolation of $((kT/n,z_{kT/n}):k=0,\ldots,n)$ for any function $z:[0,T]\to\R$.
It is immediate from the Lipschitz continuity of $f$ that it has a linear growth, and consequently that
\begin{equation}\label{2.1}
\E_{\mathbb P}[\exp\left({\alpha f^n(X)}\right)]<\infty, \qquad  \forall\alpha\in\mathbb R.
\end{equation}
The investor will assess the quality of a hedge by the resulting expected utility. Assuming exponential utility with constant absolute risk aversion $\lambda>0$, the {\it utility indifference price} and the {\it certainty equivalent price} of
the claim $f^n(X)$ do not depend on the investor's initial wealth
and, respectively, take the well-known forms
\begin{equation*}
\pi(n,\lambda):=
\frac{1}{\lambda}\log\left(\frac{\inf_{(\gamma,\delta)\in\calA^n}\E_{\mathbb P}\left[\exp\left(-\lambda\left(V^{\gamma,\delta}_T-f^n(X)\right)\right)\right]}
{\inf_{(\gamma,\delta)\in\calA^n}\E_{\mathbb P}\left[\exp\left(-\lambda V^{\gamma,\delta}_T\right)\right]}\right)
\end{equation*}
and
\begin{equation}\label{eq_c}
c(n,\lambda):=\frac{1}{\lambda}
\log\left(\inf_{(\gamma,\delta)\in\calA^n}\E_{\mathbb P}\left[\exp\left(-\lambda\left(V^{\gamma,\delta}_T-f^n(X)\right)\right)\right]\right).
\end{equation}
Note that
\begin{align}\label{eq_pi_c}
\pi(n,\lambda)=c(n,\lambda)-\frac{1}{\lambda}
\log\left(\inf_{(\gamma,\delta)\in\calA^n}\E_{\mathbb P}\left[\exp\left(-\lambda V^{\gamma,\delta}_T\right)\right]\right).
\end{align}

The following scaling limits are the main results of the paper. The proof is given at the end of Section \ref{sec:3}.
\begin{thm}\label{thm.1}
For $n\rightarrow\infty$ and re-scaled high risk-aversion $n\ell$ with $\ell>0$ fixed, the certainty equivalent price and the utility indifference price of $f^n(X)$ have the following scaling limit
\begin{align}\notag
 \lim_{n\rightarrow\infty} c(n,n\ell)&=\lim_{n\rightarrow\infty} \pi(n,n\ell)\\\label{2.3}
 &=\pi(\ell):=
 \sup_{\nu\in\calV}\E_{\mathbb P}\Big[f\left(X^{(\nu)}\right)- \frac{1}{2\ell{T}}\int_{0}^T \pen\left(\frac{\nu_t}{\sigma^2}\right)dt \Big],
\end{align}
where
\begin{align*}
X^{(\nu)}_t&:=X_0+\int_{0}^t\sqrt{\nu}_udW_u, \ \ t\in [0,T],\\
\pen(y)&:=y-\log y-1, \qquad  y>0.
\end{align*}
and
$\calV$ is the class of all bounded, non negative
$\mathcal(\mathcal F^W_t)$-predictable processes $\nu$.
\end{thm}
Let us finish this section with the following three remarks.
\begin{rem}
The scaling (\ref{2.scale2}) is a technical one and needed for the proof of the upper bound. More precisely,
(\ref{2.scale2}) is used in Lemma \ref{lem:main2}
where we implement a strong invariance principle from
\cite{dol-2012}. The question whether this scaling can be removed and the result will still hold true is an interesting question. We expect that the proof of such a result would require additional machinery from dynamical programming and non-linear partial differential equations.  We leave this challenging question for future research.

The scaling (\ref{2.scale1}) is used in the proof of the lower bound. Let us notice that if we allow power options (with power $3$)
and (\ref{2.scale1}) does not hold, i.e.
$\lim_{n\rightarrow\infty} n^{3/2} h(n)<\infty$ then there will be an additional constraint on the process $\nu$, which appears on the right-hand side of (\ref{2.3}). As a result, in general, the scaling limit of the utility
indifference prices will be less than or equal to the one given by
(\ref{2.3}).
Since, we added power options from technical reasons we assume the scaling (\ref{2.scale2}), which says that these options
are not too cheap.
\end{rem}

\begin{rem}
Observe that if we take $\ell$ to infinity then we have
$$\lim_{\ell\rightarrow\infty}\pi(\ell)=\sup_{\nu\in\calV}\E_{\mathbb P}\left[f\left(X^{(\nu)}\right)\right].$$
The above right-hand side can be viewed as the model-free option price, see \cite{GHLT:14}. This corresponds to the case where
the investor wants to super-replicate the payoff $f(X)$
without any assumptions on the volatility. For the case where $\ell\rightarrow 0$
it is straightforward to check that
$$\lim_{\ell\rightarrow 0} \pi(\ell)= \E_{\mathbb P}\left[f\left(X^{(\sigma)}\right)\right],$$
where
$$X^{(\sigma)}_t=X_0+\sigma W_t, \ \ t\in [0,T].$$
Namely, we converge to the unique price of the continuous time complete market given by (\ref{2.bac}).
\end{rem}
\begin{rem}
A natural question is whether Theorem \ref{thm.1} can be extended to the case where the risky asset $(X_t)_{t\in [0,T]}$ is given by a
geometric Brownian motion, in other words, whether our scaling limit is valid for the Black--Scholes model.

The immediate conjecture is that for the Black--Scholes given by
$$\frac{d\mathrm{X}_t}{\mathrm{X}_t}=\sigma dW_t+\mu dt,$$
the scaling limit of the utility indifference prices takes the from
$$\sup_{\nu\in\calV}\E_{\mathbb P}\Big[f\left(\mathrm{X}^{(\nu)}\right)-
 \frac{1}{2\ell{T}}\int_{0}^T \pen\left(\frac{\nu_t}{\sigma^2}\right)dt \Big],
$$
where
$$\mathrm{X}^{(\nu)}_t=X_0\exp\left(\int_{0}^t\nu_s dW_s-\frac{1}{2}\int_{0}^t\nu^2_s ds\right), \ \ t\in [0,T].$$
The proof of this conjecture is far from obvious. First, in our duality result
Proposition \ref{prop.dual}
we need to assume (see \eqref{exp})
that the exponential moments of $X$ are exists.
This is no longer the case for the geometric Brownian motion and so even the duality
requires additional ideas.
The second difficulty is to formulate and prove the
``correct" analog of Lemmas \ref{lem:main1}--\ref{lem:main2}.
At this stage we leave the Black--Scholes setup for future research.
\end{rem}

\section{A dual representation and a scaling limit}\label{sec:3}
 Set $n\in\mathbb N$. Denote by $\calQ^n$ the set of all probability measures $\mathbb Q\sim\mathbb P$ with finite entropy
$$\E_{\mathbb Q}\left[\log\frac{d\mathbb Q}{d\mathbb P}\right]<\infty$$
relative to $\mathbb P$ and such that
the processes
$$(X_{kT/n})_{0\leq k\leq n} \ \ \mbox{and}  \ \ \left\{\sum_{i=0}^{k-1} |X_{(i+1)T/n}-X_{iT/n}|^3-kh(n)\right\}_{0\leq k\leq n}$$
are, respectively, $\mathbb Q$-martingale and a $\mathbb Q$-supermartingale, with respect to the filtration
$\{\calF^W_{kT/n}\}_{0\leq k\leq n}$.
Denote by $\hat{\mathbb Q}$ the unique probability measure such that
 $\hat{\mathbb Q}\sim\mathbb P$ and $(X_t)_{t\in[0,T]}$ is $\hat{\mathbb Q}$-martingale.
From (\ref{2.scale1}) it follows that for sufficiently large $n$, we have
$\hat{\mathbb Q}\in\calQ^n$ and so $\calQ^n\neq\emptyset$.

%\subsection{An equivalent form for the dual problem}

%We now rewrite the right-hand side of (\ref{dual}) by using the explicit expression for
We now represent the Radon--Nykodim derivative using Girsanov's kernels. For any
probability measure $\mathbb Q\sim\mathbb P$, let $\psi^{\mathbb Q}$ be an $(\calF_t^W)$-progressively measurable process such that $\int_0^T|\psi^{\mathbb Q}_s|^2ds<\infty$, $\PP$-a.s.~ and
\begin{align}\label{eq_RN}
\frac{d\Q}{d\PP}|{\calF_t^W}=\exp\Big(\int_0^t\psi^{\mathbb Q}_sdW_s-\frac{1}{2}\int_0^t|\psi^{\mathbb Q}_s|^2ds\Big),\qquad t\in[0,T],
\end{align}
We refer to $\psi^{\mathbb Q}$ as the {\it Girsanov kernel} associated with $\Q$.
The dynamics of $X$ can be written equivalently as
\begin{align}\notag
X_t=X_0+\sigma W^{\mathbb Q}_t+\sigma\int_{0}^t \psi^{\mathbb Q}_s ds+\mu t,\qquad t\in[0,T],
\end{align}
where $W^\Q_t:=W_t-\int_0^t\psi^{\mathbb Q}_sds$ is a Wiener process under $\Q$.
%We obtain that for any $n\in\mathbb N$ and $\lambda>0$
%\begin{align}\label{WQ}
%\sup_{\mathbb Q\in\calQ^n}\mathbb E_{\mathbb Q}\left[f^n(X)-\frac{1}{\lambda}\log\left(\frac{d\mathbb Q}{d\mathbb P}\right)\right]=
%\sup_{\mathbb Q\in\calQ^n}\mathbb E_{\mathbb Q}\left[f^n(X)-\frac{1}{2\lambda}\int_0^T|\psi^{\mathbb Q}_s|^2ds\right].
%\end{align}

We arrive at the
dual representation for the certainty equivalent of $f^n(X)$, given in \eqref{eq_c}.
%which is defined by
%$$\frac{1}{\lambda}
%\log\left(\inf_{(\gamma,\delta)\in\calA^n}\E_{\mathbb P}\left[\exp\left(-\lambda\left(V^{\gamma,\delta}_T-f^n(X)\right)\right)\right]\right)$$
%where $\lambda>0$ is the risk aversion of the exponential utility.
Although this dual representation is quite standard (under the appropriate
 growth conditions), since we could not find
a direct reference, we provide a self contained proof.
\begin{prop}\label{prop.dual}
Let $n\in\mathbb N$ be sufficiently large such that $\calQ^n\neq\emptyset$ and fix an arbitrary $\lambda>0$.
Then,
\begin{align}
\begin{split}
\label{dual}
%&\frac{1}{\lambda}
%\log\left(\inf_{(\gamma,\delta)\in\calA^n}\E_{\mathbb P}\left[\exp\left(-\lambda\left(V^{\gamma,\delta}_T-f^n(X)\right)\right)\right]\right)\\
%&\quad=
%\sup_{\mathbb Q\in\calQ^n}\mathbb E_{\mathbb Q}\left[f^n(X)-\frac{1}{\lambda}\log\left(\frac{d\mathbb Q}{d\mathbb P}\right)\right]\\
%&\quad=
c(n,\lambda)=\sup_{\mathbb Q\in\calQ^n}\mathbb E_{\mathbb Q}\left[f^n(X)-\frac{1}{2\lambda}\int_0^T|\psi^{\mathbb Q}_s|^2ds\right].
\end{split}
\end{align}
\end{prop}
\begin{proof}
Since for any $(\gamma,\delta)\in \calA^n$ and $\lambda>0$ we have
$V^{\lambda \gamma,\lambda \delta}_{kT/n}=\lambda V^{\gamma,\delta}_{kT/n}$ for all $k=0,1,...,n$, then
without loss of generality we take $\lambda=1$.
As usual, the proof rests on the classical Legendre--Fenchel duality inequality
\begin{equation}\label{classic}
xy \leq e^x + y(\log y-1), \quad x\in\R,\; y>0, \qquad \text{ with `$=$' iff } y=e^x.
\end{equation}
We start with proving the inequality `$\geq$' in \eqref{dual}.

Let $(\gamma,\delta)\in\calA^n$ that satisfies
$$\E_{\mathbb P}\left[\exp\left(-\left(V^{\gamma,\delta}_T-f^n(X)\right)\right)\right]<\infty.$$
Choose an arbitrary $\mathbb Q\in\calQ^n$.
From the
 Cauchy--Schwarz inequality and
(\ref{2.1}) it follows that
$\E_{\mathbb P}[e^{-V^{\gamma,\delta}_T/2}]<\infty$.
This together with (\ref{classic})
for  $x=\max(0,-V^{\gamma,\delta}_T/2)$ and $y=\frac{d\mathbb Q}{d\mathbb P}$, gives that
$$
\E_{\mathbb Q}\left[\max\left(0,-V^{\gamma,\delta}_T\right)\right]
=2\E_{\mathbb P}\left[\frac{d\mathbb Q}{d\mathbb P}\max\left(0,-V^{\gamma,\delta}_T/2\right)\right]<\infty.
$$
From (\ref{portfolio}) it follows that the portfolio value process $\{V^{\gamma,\delta}_{kT/n}\}_{0\leq k\leq n}$
is a local $\mathbb Q$-supermartingale, and so from
\cite[Proposition 9.6]{FollmerSchied:16}
and the inequality $\E_{\mathbb Q}\big[\max(0,-V^{\gamma,\delta}_T)\big]<\infty$ we get
$\E_{\mathbb Q}[V^{\gamma,\delta}_T]\leq 0$.

To conclude our claim we can use (\ref{classic}) again to show that for any $z>0$, the following estimates hold
\begin{align*}
&\E_{\mathbb P}\left[\exp\left(-\left(V^{\gamma,\delta}_T-f^n(X)\right)\right)\right]\\
&\quad\ge\E_{\mathbb P}\left[\left(e^{- V^{\gamma,\delta}_T}+ V^{\gamma,\delta}_T
z e^{-f^n(X)} \frac{d\mathbb Q}{d\mathbb P}\right)e^{ f^n(X)}\right] \\
&\quad\geq  \E_{\mathbb P}\left[\left(-z e^{- f^n(X)}\frac{d\mathbb Q}{d\mathbb P}\left(\log
\left(ze^{- f^n(X)} \frac{d\mathbb Q}{d\mathbb P}\right)-1\right)\right)e^{ f^n(X)}\right] \\
&\quad=-z\left(\log z -1\right)+z\E_{\mathbb Q}\left[f^n(X)-\log \frac{d\mathbb Q}{d\mathbb P}\right].
\end{align*}
Indeed, the first inequality follows since $\E_{\mathbb Q}[V^{\gamma,\delta}_T]\leq 0$ and the second inequality follows by \eqref{classic} with the choice $x=-V^{\gamma,\delta}_T$ and $y=ze^{-f^n(X)} d\mathbb Q/d\mathbb P$. Finally, taking supermum over $z$ in the last expression, one obtains from
 (\ref{classic}) that the supremum is $\exp(\E_{\mathbb Q}[f^n(X)-\log (d\mathbb Q/d\mathbb P)])$.
 Together with \eqref{eq_RN} it implies the inequality `$\geq$' in \eqref{dual}.

Next, we prove the reversed inequality `$\leq$' in \eqref{dual}.
Define the probability measure $\hat{\mathbb P}$ by
\begin{equation*}
\frac{d \hat{\mathbb P}}{d\mathbb P}=\frac{e^{ f^n(X)}}{\E_{\mathbb P}[e^{f^n(X)}]}.
\end{equation*}
Without loss of generality we can assume that the right-hand side of
(\ref{dual}) is finite. Thus for any $\mathbb Q\in\calQ^n$
\begin{equation*}
\E_{\mathbb Q}\left[\log  \frac{d\mathbb Q}{d \hat{\mathbb P}}\right]=\E_{\mathbb Q}\left[\log \frac{d\mathbb Q}{d\mathbb P }-f^n(X)\right]+
\log\left(\E_{\mathbb P}[e^{f^n(X)}]\right).
\end{equation*}
Next, we show that the supremum in (\ref{dual}) is attained.

From the well-known Komlos-argument, see e.g., \cite[Lemma A1.1]{DS94}, we obtain a maximizing sequence $\{\mathbb Q_m\}_{m\in\N}\subset \calQ^n$, for which $Z_m:=d\mathbb Q_m/d\hat{\mathbb P}$ converges almost surely as $m\to\iy$. Without loss of generality, $\{H(\mathbb Q_m|\hat{\mathbb P})\}_{m\in\N}$ can be assumed to be bounded, where
$$H(\mathbb Q_m|\hat{\mathbb P}):=\E_{\hat{\mathbb P}}\left[Z_m\log Z_m\right]=
\E_{\mathbb Q_m}\left[\log \frac{d\mathbb Q_m}{d\hat{\mathbb P}}\right].$$
Observe that the function $y\rightarrow \frac{1}{e}+ y\log y$ is non negative. Hence,
$$\lim_{M\rightarrow\infty}\sup_{m\in\mathbb N}\mathbb E_{\hat{\mathbb P}}\left[Z_m \one_{\{Z_m>M\}}\right]\leq
\lim_{M\rightarrow\infty}\frac{1}{\log M}\sup_{m\in\mathbb N}\left(\frac{1}{e}+H(\mathbb Q_m|\hat{\mathbb P})\right)=0$$
where $\one_{\cdot}$ equals $1$ on the event $\cdot$ and $0$ otherwise.
Thus,
$(\hat{\mathbb P}\circ (Z_m)^{-1})_{m\in\mathbb N}$ are uniformly integrable, and so
the
convergence holds also in $L^1(\hat {\mathbb P})$.
We conclude that $Z_0:= \lim_{m\rightarrow\infty} Z_m$ yields the density (Radon-Nikodym derivative with respect to $\hat{\mathbb P}$)
 of a probability measure $\mathbb Q_0$. From Fatou's lemma (the function $y\rightarrow y\log y$ is bounded from below) we get $H(\mathbb Q_0|\hat{\mathbb P})\leq \lim_{m\rightarrow\infty} H(\mathbb Q_m|\hat{\mathbb P})$ and
\begin{align*}
&\E_{\mathbb Q_0}\left[|X_{(k+1)T/n}-X_{kT/n}|^3\big|\mathcal F_{kT/n}\right]\\
&\quad\leq\lim\inf_{m\rightarrow\infty}
\E_{\mathbb Q_m}\left[|X_{(k+1)T/n}-X_{kT/n}|^3\big|\mathcal F_{kT/n}\right]\leq h(n), \qquad \forall k\leq n.
\end{align*}
So, attainment of the supremum in (\ref{dual}) is proven if we can argue that $\mathbb Q_0$
is a martingale measure for $(X_{kT/n})_{0\leq k\leq n}$. To this end, it is sufficient to argue that for any $k$ the measures $(\mathbb Q_m\circ (X_{kT/n})^{-1})_{m\in\mathbb N}$
are uniformly integrable in the sense that
\begin{align}\label{intun}
\lim_{M\to\iy}\;\sup_{m\in\mathbb N}\E_{\mathbb Q_m}\left[|X_{kT/n}|\one_{\{|X_{kT/n}|>M\}}\right]=0.
\end{align}

From the symmetry of the Brownian motion, the simple inequality
$$\exp\left(\sup_{0\leq t \leq T} |z_t|^p\right)\leq \exp\left(|\sup_{0\leq t \leq T} z_t|^p\right)+\exp\left(|\inf_{0\leq t \leq T} z_t|^p\right), \ \  \forall z\in \calC[0,T] $$
and the fact that (under $\mathbb P$)
$\sup_{0\leq t\leq T}W_t\sim |W_T|$ we obtain
\begin{align}\label{exp}
&\mathbb E_{\mathbb P}\left[\exp\left(\sup_{0\leq t\leq T}|X_t|^p\right)\right]\leq 2
\mathbb E_{\mathbb P}\left[\exp\left(\left(X_0+|\mu| T+\sigma |W_T|\right)^p\right)\right]\nonumber\\
&=
4\int_{0}^{\infty} \exp\left(\left(X_0+|\mu| T+\sigma x\right)^p\right)\frac{\exp(-x^2/2)}{\sqrt{2\pi}}dx
<\infty,  \quad \forall p\in (0,2).
\end{align}
By applying the H{\"o}lder's inequality, (\ref{2.1}) and (\ref{exp})
it follows that for any $k$,
$\E_{\hat{\mathbb P}}[\exp(|X_{kT/n}|^{3/2})]<\infty$.
Hence,
from (\ref{classic})
\begin{align*}
\sup_{m\in\mathbb N}\E_{\mathbb Q_m}\left[|X_{kT/n}|^{3/2}\right]
&=
\sup_{m\in\mathbb N}\E_{\hat{\mathbb P}}\left[Z_m|X_{kT/n}|^{3/2}\right]\\
&\leq
\E_{\hat{\mathbb P}}\left[\exp(|X_{kT/n}|^{3/2})\right]+\sup_{m\in\mathbb N}H(\mathbb Q_m|\hat{\mathbb P})<\infty
\end{align*}
and (\ref{intun}) follows.

Now, we arrive at the final step of the proof.
 We follow the approach in \cite{Fritelli:00}. Indeed, the perturbation argument
 for the proof of Theorem 2.3 there shows that the entropy minimizing measure's density is of
 the form
$$Z_0=\frac{e^{-\xi_0}}{\E_{\hat{\mathbb P}}[e^{-\xi_0}]}$$ for some random variable
 $\xi_0$ with $\E_{\mathbb Q_0}[\xi_0]=0$ and $\E_{\mathbb  Q}[\xi_0]\leq 0$
 for any $\mathbb Q \in \calQ^n$.
 The separation argument for \cite[Theorem 2.4]{Fritelli:00} shows that $\xi_0$ is contained in the $L^1(\mathbb Q)$-closure of
 $\{V^{\gamma,\delta}_T: \; (\gamma,\delta)\in\calA^n\}-L^{\infty}_{+}$ for any $\mathbb Q \in \calQ^n$, where $L^{\infty}_{+}$ is the set of all non negative random variables, which are uniformly bounded.
 Since $\calQ^n\neq\emptyset$, Lemma 3.1 in \cite{Napp:03} yields that $\xi_0$ must be of the same form $\xi_0=V^{\gamma_0,\delta_0}_T-R_0$ for
  some $(\gamma_0,\delta_0)\in\calA^n$ and some $R_0 \geq 0$.
As a result, we may bound the left-hand side in (\ref{dual}) as follows
\begin{align*}
&\log\left(\inf_{(\gamma,\delta)\in\calA^n}\E_{\mathbb P}\left[\exp\left(-\left(V^{\gamma,\delta}_T-f^n(X)\right)\right)\right]\right)\\
&\quad\le\log\left(\E_{\mathbb P}[e^{f^n(X)}]\right)+\log \left(\E_{\hat{\mathbb P}}[\exp(-V^{\gamma_0,\delta_0}_T)]\right)\\
&\quad\leq
\log\left(\E_{\mathbb P}[e^{f^n(X)}]\right)+\log \left(\E_{\hat{\mathbb P}}[\exp(-\xi_0)]\right)\\
&\quad=\log\left(\E_{\mathbb P}[e^{f^n(X)}]\right)+\E_{\mathbb Q_0}
\left[\xi_0+\log \left(\E_{\hat{\mathbb P}}[\exp(-\xi_0)]\right)\right]\\
&\quad=
\log\left(\E_{\mathbb P}[e^{f^n(X)}]\right)- \E_{\mathbb Q_0} \left[\log Z_0\right]\\
&\quad=\sup_{\mathbb Q\in\calQ^n}\E_{\mathbb Q}\left[f^n(X)-\frac{1}{\lambda}\log \frac{d\mathbb Q}{d\mathbb P}\right].
\end{align*}
The first inequality follows by the definition of $\hat \PP$; the second inequality uses $R_0\ge 0$; the first equality is derived by $\E_{\Q_0}[\xi_0]=0$; the second equality follows by the definition of $\xi_0$; and finally, the last equality is deduced by our choice of $\Q_0$ as a measure that attains the supremum over $\calQ^n$.

 Together with \eqref{eq_RN} we obtain the inequality `$\leq$' in \eqref{dual}.

\qed
\end{proof}

The main technical challenge in this paper is showing that the scaling limit of the value function of the {\it drift control problem} from the right-hand side of \eqref{dual} equals the value of the {\it volatility control problem} from the right-hand side of \eqref{2.3}. This is summarized in the next theorem, whose proof is given in the next Section.
\begin{thm}\label{thm.2}
The following scaling limit holds:
\begin{align}
\notag
& \lim_{n\rightarrow\infty}
 \sup_{\mathbb Q\in\calQ^n}\mathbb E_{\mathbb Q}\Big[f^n(X)-\frac{1}{2\lambda}\int_0^T|\psi^{\mathbb Q}_s|^2ds\Big]
 \\\label{2.33}
 &\quad=
 \sup_{\nu\in\calV}\E_{\mathbb P}\Big[f\left(X^{(\nu)}\right)- \frac{1}{2\ell{T}}\int_{0}^T \pen\left(\frac{\nu_t}{\sigma^2}\right)dt \Big].
\end{align}
\end{thm}

We end this section with proving Theorem \ref{thm.1}.
\begin{proof}[Proof of Theorem \ref{thm.1}]
Proposition \ref{prop.dual} and Theorem \ref{thm.2} imply that
\begin{align}\notag
\lim_{n\rightarrow\infty} c(n,n\ell)= \sup_{\nu\in\calV}\E_{\mathbb P}\Big[f\left(X^{(\nu)}\right)- \frac{1}{2\ell{T}}\int_{0}^T \pen\left(\frac{\nu_t}{\sigma^2}\right)dt \Big].
\end{align}
Utilizing \eqref{eq_pi_c}, the proof that $\lim_{n\rightarrow\infty} c(n,n\ell)=\lim_{n\rightarrow\infty} \pi(n,n\ell)$ follows once we show that
\begin{align}\label{bound2}
\lim_{n\to\infty}\frac{1}{n\ell}
\log\left(\inf_{(\gamma,\delta)\in\calA^n}\E_{\mathbb P}\left[\exp\left(-n\ell V^{\gamma,\delta}_T\right)\right]\right)=0.
\end{align}
We prove this in two steps. First, notice that we can use $(\gamma,\delta)\equiv (0,0)$ to get an upper bound as follows:
\begin{align}\label{bound3}
\inf_{(\gamma,\delta)\in\calA^n}\E_{\mathbb P}\left[\exp\left(-n\ell V^{\gamma,\delta}_T\right)\right]
\le
\E_{\mathbb P}\left[\exp\left(-n\ell V^{0,0}_T\right)\right]=1.
\end{align}
This establishes the relation `$\le$' in \eqref{bound2}, replacing `$=$'.
Next, we show that \eqref{bound2} holds with `$\ge$' instead of equality, which together with the last staement, finishes the proof.

Fix $\ell>0$.
%By passing to a subsequence (which is still denoted by $n$) we assume without loss of generality that the limit in \eqref{bound2} exists and is less than $\infty$ (otherwise the statement is obvious).
%Moreover,
We assume that
$n$ is sufficiently large such that
$\hat{\mathbb Q}\in\mathcal Q^n$, where,
 recall that $\hat{\mathbb Q}$
is the unique martingale measure for the continuous time Bachelier model.

Let $(\gamma,\delta)\in\mathcal A^n$ be such that
$\E_{\mathbb P}[\exp(-n\ell V^{\gamma,\delta}_T)]<\infty$. This can be assumed without loss of generality due to \eqref{bound3}.
Using the same arguments as in the proof of
Proposition \ref{prop.dual} we get
$\mathbb E_{\hat{\mathbb Q}}[V^{\gamma,\delta}_T]\leq 0$.
Thus, from Jensen's inequality for the convex function
$y\rightarrow e^{-n\ell y/2}$ and the Cauchy--Schwarz inequality,  we obtain
$$1\leq \E_{\hat{\mathbb Q}}\left[\exp\left(-n\ell V^{\gamma,\delta}_T/2\right)\right]
\leq\left(\E_{\mathbb P}\left[\exp\left(-n\ell V^{\gamma,\delta}_T\right)\right]\right)^{1/2}
\left(\E_{\mathbb P}\Big[\Big(\frac{d\hat{\mathbb Q}}{{d\mathbb P}}\Big)^2\Big]\right)^{1/2},
$$
and so $\E_{\mathbb P}[\exp(-n\ell V^{\gamma,\delta}_T)]$
is uniformly bounded from below. Thus, we obtain that \eqref{bound2} holds with `$\ge$' instead of equality.
%from (\ref{2.2}), (\ref{dual}), and (\ref{WQ}) we get that there exists
%a sequence of probability measures $\Q^n\in\mathcal Q^n$, $n\in\mathbb N$
%such that
%\begin{align}\label{new}
%\lim_{n\rightarrow\infty}\pi(n,n\ell)\leq
%\lim_{n\rightarrow\infty}
%\mathbb E_{\Q^n}\left[f^n(X)-\frac{1}{2n\ell}\int_0^T|\psi^{\Q^n}_t|^2 dt\right].
%\end{align}
\end{proof}

\section{Proof of Theorem \ref{thm.2}}\label{sec:4}
The proof relies on two bounds, which are provided in two separate subsections.

\subsection{Upper bound}
This section is devoted to the proof of the inequality
`$\leq$' in \eqref{2.33}.
We start with the following lemma which provides an intuition for the penalty term
$\pen(y):=y-\log y-1$, $y>0$.

\begin{lem}\label{lem:main1}
Consider a filtered probability space $(\bar\Omega,\bar\calF,\{\bar\calF_t\}_{t\in[0,T]},\bar\PP)$
which supports an $(\bar\calF_t)$-Wiener process $\bar W$. As usual we assume that the filtration
  $(\bar\calF_t)$ is right-continuous and contains the null sets. Then, for any time instances $0\le t_1<t_2\le T$ and every $(\bar\calF_t)$-progressively measurable process $\psi$, satisfying the conditions $\int_{t_1}^{t_2}\psi^2_udu<\iy$ and $\bar\E[\int_{t_1}^{t_2}\psi_udu\mid\calF_{t_1}]=0$, one has,
\begin{align}\label{eq:lem}
\bar\E\Big[\int_{t_1}^{t_2}\psi^2_udu\;\big|\;\calF_{t_1}\Big]
\ge
\pen\Big(\frac{1}{t_2-t_1}\bar\E\Big[\Big(\bar W_{t_2}-\bar W_{t_1}+\int_{t_1}^{t_2}\psi_udu\Big)^2\;\big|\;\calF_{t_1}\Big]\Big).
\end{align}
Moreover, set the parameterized (by $\beta$) processes $(\theta^\beta_t)_{t\in[t_1,t_2]}$
and $(\vartheta^\beta_t)_{t\in[t_1,t_2]}$
by
\begin{align}\label{theta}
&\theta^\beta_t:=\int_{t_1}^t(\beta(t_2-t_1)-(t_2-s))^{-1}d\bar W_s,\ \ \beta>1 \\
&\vartheta^{\beta}_t:=-\int_{t_1}^t(\beta(t_2-t_1)+(t_2-s))^{-1}d\bar W_s,\ \ \beta>0.
\label{vartheta}
\end{align}
Then \eqref{eq:lem} holds with equality for $\psi_t$, $t\in [t_1,t_2]$, given by
\begin{align}
\notag
\psi_t&=\theta^{\bar{\beta}}_t\one_{\Big\{\E\Big[\Big(\bar W_{t_2}-\bar W_{t_1}+\int_{t_1}^{t_2}\psi_udu\Big)^2\;\big|\;\calF_{t_1}\Big]>t_2-t_1\Big\}}\\\label{newnewnew}
&\quad+
\vartheta^{\bar{\beta}}_t\one_{\Big\{\E\Big[\Big(\bar W_{t_2}-\bar W_{t_1}+\int_{t_1}^{t_2}\psi_udu\Big)^2\;\big|\;\calF_{t_1}\Big]<t_2-t_1\Big\}},
\end{align}
where
\begin{align}\label{newnew}
\bar{\beta}=\frac{\E\Big[\Big(\bar W_{t_2}-\bar W_{t_1}+\int_{t_1}^{t_2}\psi_udu\Big)^2\;\big|\;\calF_{t_1}\Big]}{\Big|\E\Big[\Big(\bar W_{t_2}-\bar W_{t_1}+\int_{t_1}^{t_2}\psi_udu\Big)^2\;\big|\;\calF_{t_1}\Big]-(t_2-t_1)\Big|}.
\end{align}
\end{lem}
\begin{proof}
By the independent increments and the scaling property of Wiener process, we assume without loss of
generality that $t_1=0$, $t_2=1$ and $\calF_{0}$ is the trivial $\sigma$-algebra.

Obviously, (\ref{eq:lem}) holds trivially if $\bar\E[\int_0^1\psi^2_tdt]=\iy$,
hence, for the rest of the proof we also assume that $\bar\E[\int_0^1\psi^2_tdt]<\iy$.
This in turn implies that $\bar\E[\psi_t^2]<\infty$ for almost every $t\in[0,1]$, with respect to Lebesgue measure.
Thus, without loss of generality, we may assume that
$\bar\E[\psi_t]=0$ for any $t\in[0,1]$. Indeed,  set up $\bar\psi_t:=\psi_t-\bar\E[\psi_t], t\in[0,1]$. Then, if we prove \eqref{eq:lem} for $\bar\psi$, then
\begin{align}\notag
\bar\E\Big[\int_0^1\psi^2_tdt\Big]
&\ge
\bar\E\Big[\int_0^1\bar\psi_u^2dt\Big]\\\notag
&
\ge
\pen\Big(\bar\E\Big[\Big(\bar W_1+\int_0^1\bar\psi_tdt\Big)^2\Big]\Big)\\\notag
&=
\pen\Big(\bar\E\Big[\Big(\bar W_1+\int_0^1\psi_tdt\Big)^2\Big]\Big),
\end{align}
where the second inequality follows from \eqref{eq:lem} applied to $\bar\psi$ and the equality follows since
$\int_{0}^{1}\bar\E[\psi_t]dt=0$.

Next, by applying standard density arguments in
$L^2(dt\otimes\bar{\mathbb P})$ we can assume that $\psi$ is a simple process (see Section 3.2 in \cite{kar-shr})
in the sense that $\psi$ is bounded and
there exists a deterministic partition $0=t_0<t_1<...<t_m=1$ such that $\psi$ is a (random) constant
on each interval $(t_i,t_{i+1}]$.
Hence, for the rest of the proof we fix
a simple process $\psi$ which satisfies
$\bar\E[\psi_t]=0$, for all $t\in[0,1]$.

We split the proof into two cases: (I) $\bar \E[(\bar W_1+\int_0^1\psi_t dt)^2]>1$; and (II) $\bar \E[(\bar W_1+\int_0^1\psi_t dt)^2]<1$. When the expected value equals $1$, \eqref{eq:lem} follows immediately since $\pen(1)=0$.

\textbf{Case I: $\bar \E[(\bar W_1+\int_0^1\psi_t dt)^2]>1$}.
Let $\bar\beta$ be given by (\ref{newnew}).
Observe that
$\bar \E[(\bar W_1+\int_0^1\psi_t dt)^2]=\bar{\beta}/(\bar{\beta}-1)$. In order to prove the inequality
(\ref{eq:lem}) it is sufficient to show that
 \begin{align}\notag
 \bar \E\Big[\Big(\bar W_1+\int_0^1\psi_t dt\Big)^2-\bar{\beta}\int_0^1\psi^2_tdt\Big]
&\leq
\frac{\bar{\beta}}{\bar{\beta}-1}-\bar{\beta}\pen\Big(\frac{\bar\beta}{\bar\beta-1}\Big)
\\\label{lag}
&=\bar\beta\log\frac{\bar\beta}{\bar\beta-1}.
\end{align}
Let $(\mathcal F^{\bar W}_t)_{t\in [0,1]}$
be the augmented filtration generated by the Brownian motion $(\bar W_t)_{t\in [0,1]}$
and let $(u_t)_{t\in [0,1]}$ be the optional projection of $\psi$ on $(\mathcal F^{\bar W}_t)_{t\in [0,1]}$ (exists since $\psi$ is bounded).
Set, $v_t:=\psi_t-u_t$, $t\in [0,T]$.
Clearly, $\bar W_{[t,1]}-\bar W_t$ is independent of $\psi_t$ and $\bar W_{[0,t]}$, where $\bar W_{[a,b]}$ is the restriction of $\bar W$ to the interval $[a,b]$.
This together with the fact that
$\mathcal F^{\bar W}_1$ is generated by $\mathcal F^{\bar W}_t$ and $\bar W_{[t,1]}-\bar W_t$ yields
$u_t:=\bar \E[\psi_t|\mathcal F^{\bar W}_t]=
\bar \E[\psi_t|\mathcal F^{\bar W}_1]$ for all $t$.
Thus,
\begin{align}
&\bar \E\left[\left(\bar W_1+\int_0^1\psi_t dt\right)^2-\bar\beta\int_0^1\psi^2_tdt\right]\notag\\
&\qquad=
\bar \E\left[\left(\bar W_1+\int_0^1 u_t dt\right)^2-\bar\beta\int_0^1 u^2_tdt\right]+\bar\E\left[\left(\int_0^1 v_t dt\right)^2-\bar\beta\int_0^1 v^2_tdt\right]\notag\\
&\qquad
\leq \bar \E\left[\left(\bar W_1+\int_0^1 u_t dt\right)^2-\bar\beta\int_0^1 u^2_tdt\right],\label{1}
\end{align}
where the inequality follows from Jensen's inequality and the fact that $\bar\beta>1$.

From the martingale representation theorem and the fact that $\psi$ is simple it follows that there exists
a (jointly) measurable map $\kappa:[0,1]^2\times\Omega\rightarrow\mathbb R$ such that $\kappa_{t,s}$ is $\mathcal F^{\bar W}_{t\wedge s}$
measurable for all $t,s\in [0,1]$ and
$$u_t=\int_{0}^t \kappa_{t,s} d\bar W_s, \qquad dt\otimes\bar{\mathbb P} \ \ \mbox{a.s.}$$
For each $t\in[0,T]$, we are applying the martingale representation theorem for the random variable
$u_t$.
Utilizing the fact that $\psi$ is piecewise constant, we obtain that $\kappa$ is jointly measurable in $s$ and $t$.
This is essential in the sequel when we apply Fubini's theorem.

Define the processes
$$\zeta_s:=\int_{s}^1 \kappa_{t,s} dt, \qquad \eta_s:=\int_{s}^1 \kappa^2_{t,s}dt, \qquad s\in [0,1].$$

We get,
\begin{align*}
&\bar \E\Big[\Big(\bar W_1+\int_0^1 u_t dt\Big)^2-\bar\beta\int_0^1 u^2_tdt\Big]\\
&\quad=
\bar \E\Big[\int_0^1 \left((1+\zeta_s)^2 -\bar\beta \eta_s\right)ds\Big]\\
&\quad\leq\bar \E\Big[\int_0^1 \Big((1+\zeta_s)^2 -\frac{\bar\beta \zeta^2_s}{1-s}\Big)ds\Big]\\
&\quad\leq  \int_0^1 \frac{\bar\beta}{\bar\beta-(1-s)}ds=\bar\beta\log\frac{\bar\beta}{\bar\beta-1}.
\end{align*}
Indeed, the first equality follows
from the stochastic Fubini theorem and the
It\^o-isometry (see Chapter IV in \cite{RY}), the first inequality follows from the
Cauchy--Schwarz inequality, the second inequality follows from maximizing the quadratic pattern (for a given $s$)
$z\rightarrow (1+z)^2-\frac{\bar\beta z^2}{1-s}$ and the last equality is a simple computation.
This together with (\ref{1}) completes the proof of (\ref{lag}).

Next, recall the process $\theta^{\beta}$ given by (\ref{theta}).
Observe that for $\kappa_{t,s}:=\theta^{\bar\beta}_s$, with $t,s\in [0,1]$ the above two inequalities are in fact equalities.
Moreover, it easy to check that
$$\bar \E\Big[\Big(\bar W_1+\int_0^1 u_t dt\Big)^2\Big]
=\int_0^1 \left(1+(1-s)\theta^{\bar\beta}_s\right)^2ds=\frac{\bar\beta}{\bar\beta-1}$$
and so for $\psi=\theta^{\bar\beta}$ we have an equality in (\ref{eq:lem}).

\textbf{Case II: $\bar \E[(\bar W_1+\int_0^1\psi_t dt)^2]<1$}.
Let $\bar\beta$ be given by (\ref{newnew}).
Observe that
$\bar \E[(\bar W_1+\int_0^1\psi_t dt)^2]=\bar\beta/(\bar\beta+1).$
 In order to prove the inequality
(\ref{eq:lem}) it sufficient to show that
\begin{align}\notag
\bar \E\Big[\Big(\bar W_1+\int_0^1\psi_t dt\Big)^2+\bar\beta\int_0^1\psi^2_tdt\Big]
&\geq
\bar\beta/(\bar\beta+1)+\bar\beta\pen\Big(\frac{\bar\beta}{\bar\beta+1}\Big)
\\\label{lag1}
&=\bar\beta\log\frac{\bar\beta+1}{\bar\beta}.
\end{align}
Let $u,v,\zeta,\eta$ defined as in Case I.
Recall the process $\vartheta^{\beta}$ given by (\ref{vartheta}).
Then, by using similar arguments as in Case I we obtain
\begin{align*}
&\bar \E\left[\left(\bar W_1+\int_0^1\psi_t dt\right)^2+\bar\beta\int_0^1\psi^2_tdt\right]\\
&\quad\geq
\bar \E\left[\left(\bar W_1+\int_0^1 u_t dt\right)^2+\bar\beta\int_0^1 u^2_tdt\right]\\
&\quad=\bar \E\Big[\int_0^1 \left((1+\zeta_s)^2 +\bar\beta \eta_s\right)ds\Big]\\
&\quad\geq\bar \E\Big[\int_0^1 \left((1+\zeta_s)^2 +\bar\beta \zeta^2_s/(1-s)\right)ds\Big]\\
&\quad\geq  \int_0^1 \frac{\bar\beta}{\bar\beta+(1-s)}ds=
\bar\beta\log\frac{\bar\beta+1}{\bar\beta}
\end{align*}
and (\ref{lag1}) follows.
Finally, we notice that for $\kappa_{t,s}:=\vartheta^{\bar\beta}_s$, $t,s\in [0,1]$ the above inequalities are in fact equalities.
In addition it is easy to check that
$$\bar \E\Big[\Big(\bar W_1+\int_0^1 u_t dt\Big)^2\Big]
=\int_0^1 \left(1+(1-s)\vartheta^{\bar\beta}_s\right)^2ds=\frac{\bar\beta}{\bar\beta+1}$$
and so for $\psi=\vartheta^{\bar\beta}$ we have an equality in (\ref{eq:lem}).\qed
\end{proof}

Next, fix $n\in\mathbb N$. The next lemma provides a bound for an expected payoff
calculated with respect to a given discrete-time martingale $M$,
which later on will stand for
$(X_{kT/n})_{0\leq k\leq n}$.
The idea is to construct a continuous-time martingale that is close in
distribution to the process $M$ on the discrete set of times and whose volatility is piecewise constant between two consecutive points on the discrete-time set.

\begin{lem}\label{lem:main2}
Let $\{M_k\}_{0\leq k\leq n}$ be a martingale defined on some probability space, with $M_0=X_0$ and that satisfies for any $k=0,\ldots, n-1$,
\begin{align}\label{asm}
 \hat\E\left[|M_{k+1}-M_k|^3\mid M_0,\ldots,M_k\right]\le h(n)
 \end{align}
where $\hat \E$ is the expectation with respect to the given probability space. Assume that for some $K>0$,
\begin{align}\label{-K}
&\hat\E\Big[f^n(M)-{\frac{1}{2n\ell}}\sum_{k=0}^{n-1}\pen\Big(\frac{n}{\sigma^2 T}
 \hat \E\left[|M_{k+1}-M_k|^2\mid M_0,\ldots,M_k\right]\Big)\Big]\ge -K
\end{align}
where
$f^n(M)=f(p^n(M))$ and with abuse of notations
$p^n(M)$ is the linear interpolation of $((kT/n,M_k):k=0,\ldots,n)$ ($p^n(M)$ is a random element in $\mathcal{C} [0,T]$).
Then, there exists a constant $C>0$ (that depends only on $\ell$, $K$, and $f$, through the Lipschitzity and the linear growth), which is independent of $n$,
such that,
\begin{align}
\notag
&\hat\E\Big[f^n(M)-{\frac{1}{2n\ell}}\sum_{k=0}^{n-1}\pen\Big(\frac{n}{\sigma^2 T}
\hat \E\left[|M_{k+1}-M_k|^2\mid M_0,\ldots,M_k\right]\Big)\Big]\\\label{1/8}
&\quad\le
 C (h(n)n)^{1/8}+\sup_{\nu\in\calV}
\E_{\mathbb P}\Big[f^n\Big(X^{(\nu)}\Big)-{\frac{1}{2\ell T}}\int_0^T\pen\Big(\frac{\nu_t}{{\sigma^2}}\Big)dt\Big].
\end{align}
\end{lem}
\begin{proof}
The Lipschitz continuity of $f$ implies that it has a linear growth. This together
with the Doob inequality for the martingale $M$, the simple
bound $\pen(y)\ge y/2-1$ and \eqref{-K} gives that
there exists a constant $\hat C>0$  (that depends only on $\ell$, $K$, and $f$, through the Lipschitzity and the linear growth), which is independent of $n$,
such that,
\begin{align}\label{C}
\hat\E\left[\max_{0\leq k\leq n}M^2_k\right]\le \hat C.
\end{align}
From Lemma 3.2 in \cite{dol-2012} and (\ref{asm}) it follows that we can
construct the martingale
$\{M_k\}_{0\leq k\leq n}$ on a new probability space
(namely the joint distribution of $(M_0,...,M_n)$ is the same as before)
which supports a sequence of identically distributed random variables $\{Y_k\}_{1\leq k\leq n}$, each has the standard normal distribution,
such that the following holds. \\
(I) For each $k=0,1,...,n-1$, $Y_{k+1}$ is independent of $\{M_i\}_{0\leq i\leq k}$ and $\{Y_i\}_{1\leq i\leq k}$.\\
(II) There exists a universal constant $\bar C>0$, which is independent of the parameters in the model for which
\begin{align}\label{MX}
\hat{\mathbb P}\left[\max_{k=0,\ldots,n}|M_k-\hat X_k|>(h(n)n)^{1/4}\right]<\bar C(h(n)n)^{1/4},
\end{align}
with
\begin{align*}
\hat X_k:=X_0+\sum_{i=0}^{k-1}Y_{i+1}\sqrt{\hat{\mathbb E}[(M_{i+1}-M_i)^2\mid M_0,\ldots,M_i]},\qquad k=0,\ldots,n.
\end{align*}
We abuse notation and use $\hat{\mathbb P}$ for the probability measure under the original space and the new space and
keep using the notation $M$ for the martingale under the new probability space.

Let us remark that in the formulation
of Lemma 3.2 in \cite{dol-2012} we have that
$\{Y_k\}_{1\leq k\leq n}$
are independent and identically distributed random variables
with the standard normal distribution such that for any $k$, $Y_{k+1}$ is independent of $\{M_i\}_{0\leq i\leq k}$.
However, the construction in the proof of this Lemma (Lemma 3.2 in \cite{dol-2012}) provides a stronger property which is property (I)
above.

Next, we use this representation to embed the law of $\{\hat X_k\}_{0\leq k\leq n}$ into the original Brownian probability space from Section \ref{sec:2}:
$(\Omega, \mathcal{F}, \{\mathcal{F}^W_t\}_{t \in [0,T]},\mathbb P)$.
By applying Theorem 1 in \cite{skorohod} we obtain that
there exist
measurable functions $\chi_k:\R^{2k+1}\to\R$, $k=1,...,n$, such that
for any $k=0,1,...,n-1$,
\begin{align}\notag
&\calL(Y_1,\ldots,Y_{k+1},M_0,\ldots,M_k,M_{k+1})\\\notag
&\quad=\calL\left(Y_1,\ldots,Y_{k+1},M_0,\ldots,M_k,\chi_{k+1}(Y_1,\ldots,Y_{k+1},M_0,\ldots,M_k,\xi)\right)
\end{align}
where $\xi$
has the standard normal distribution and is independent of $\{Y_i\}_{1\leq i\leq k+1}$ and $\{M_i\}_{0\leq i\leq k}$.

Define on the probability space $(\Omega, \mathcal{F}, \{\mathcal{F}^W_t\}_{t \in [0,T]},\mathbb P)$
the three processes $\{\bar Y_i\}_{1\leq i\leq n}$, $\{\bar \xi_i\}_{1\leq i\leq n}$
and $\{\bar M_i\}_{0\leq i\leq n}$ as follows.
First,
for any $k=1,...,n$ set
\begin{align*}
&\bar Y_{k}:=\sqrt{n/T}\big(W_{k T/n}- W_{(k-1)T/n}\big),\\
&\bar \xi_{k}:=\frac{3W_{(k-1)T/n+T/(3n)}-2W_{(k-1)T/n+T/(2n)}-W_{(k-1)T/n}}{\sqrt{\text{Var}\big( 3W_{(k-1)T/n+T/(3n)}-2W_{(k-1)T/n+T/(2n)}-W_{(k-1)T/n}\big)}}.
\end{align*}
Next, define by recursion $\bar M_0:=X_0$ and for any $k=0,1,...,n-1$,
$$\bar M_{k+1}:=\chi_{k+1}(\bar Y_1,\ldots,\bar Y_{k+1},\bar M_0,\ldots,\bar M_k,\bar\xi_{k+1}).$$
Clearly, $\{\bar Y_i\}_{1\leq i\leq n}$ and $\{\bar \xi_i\}_{1\leq i\leq n}$
have the standard normal distribution. Observe that for any $k$, $\bar Y_{k+1}$ and $\bar\xi_{k+1}$ are independent of
$W_{[0,kT/n]}$ and so, they are independent of $\{\bar Y_i\}_{1\leq i\leq k}$ and $\{\bar M_i\}_{0\leq i\leq k}$. Moreover, we notice that for any $k$,
$\bar \xi_{k}$ is independent of $\bar Y_{k}$ (they are bivariate normal and uncorrelated). Thus for any $k$,
$\bar \xi_{k+1}$ is independent of $\{\bar Y_i\}_{1\leq i\leq k+1}$ and $\{\bar M_i\}_{0\leq i\leq k}$.
From the definition of the functions
$\chi_k$, $k=1,...,n$, we conclude (by induction) that
\begin{align}\label{distri}
\calL\left(\{\bar M_i\}_{0\leq i\leq n},
\{\bar Y_i\}_{1\leq i\leq n}\right)=\calL\left(\{ M_i\}_{0\leq i\leq n},
\{Y_i\}_{1\leq i\leq n}\right).
\end{align}

Next, let
$\ph_k:\R^{k+1}\to\R_{+}$, $k=0,1,...,n-1$, be measurable functions
such that
\begin{align}\notag
\sqrt{\E[(M_{k+1}-M_k)^2\mid M_0,\ldots,M_k]}=\ph_k(M_0,\ldots,M_k), \ \ \ k=0,1,...,n-1.
\end{align}
Introduce the process $\nu\in \calV$ by
$$\nu_t:=\frac{n}{T}\sum_{k=0}^{n-1}\ph^2_k(\bar M_0,\ldots,\bar M_k)\one_{\{kT/n\leq t< (k+1)T/n\}},\qquad t\in[0,T].$$
From (\ref{distri}) it follows that the law of
$(\nu_{kT/n}:k=0,\ldots,n-1)$ equals to the law of $\left(\frac{n}{T}\E\left[\left(M_{k+1}-M_k\right)^2\mid M_0,\ldots,M_k\right]:k=0,\ldots,n-1\right).$
Since $\nu$ is constant on each of the intervals $[kT/n,(k+1)T/n]$, $k=0,\ldots,n-1$,
we conclude that
\begin{align}\notag
&\hat\E\Big[\frac{1}{n}\sum_{k=0}^{n-1}\pen\Big(\frac{n}{\sigma^2 T}
\hat \E\left[|M_{k+1}-M_k|^2\mid M_0,\ldots,M_k\right]\Big)\Big]\\\label{g}
&\qquad
=\mathbb E_{\mathbb P}\Big[\frac{1}{T}\int_0^T\pen\Big(\frac{\nu_t}{\sigma^2}\Big)dt\Big].
\end{align}
Finally, consider the process
$X^{(\nu)}$. Observe that
$$X^{(\nu)}_{kT/n}=X_0+\sum_{i=0}^{k-1}\bar Y_{i+1}\ph_i(\bar M_0,\ldots,\bar M_i), \ \ \ k=0,1,...,n.$$
This together with (\ref{distri}) yields that
$(X^{(\nu)}_{kT/n})_{0\le k\le n}$ and $(\hat X_k)_{0\le k\le n}$ have the same distribution.
Therefore, from the fact that $f\geq 0$ and Lipschitz continuous we obtain that there exists a constant
$c_1$, which does not depend on $n$, such that
\begin{align}
\notag
\hat\E\Big[f^n(M)\Big]
&\leq \mathbb E_{\mathbb P}\Big[f^n(X^{(\nu)})\Big]+c_1 (h(n)n)^{1/4}\\\notag
&\quad+\hat\E\Big[f^n(M)\one_{\left\{\max_{k=0,\ldots,n}|M_k-\hat X_k|>(h(n)n)^{1/4}\right\}}\Big]\\\label{g1}
&\leq
\mathbb E_{\mathbb P}\Big[f^n(X^{(\nu)})\Big]+C(h(n)n)^{1/8}
\end{align}
for some constant $C$
which does not depend on $n$.
The last inequality
follows from the
Cauchy--Schwarz inequality,
the linear growth of $f$, the scaling assumption (\ref{2.scale2})
and (\ref{C})--(\ref{MX}).

By combining (\ref{g})--(\ref{g1}) we complete the proof of \eqref{1/8}.
\qed
\end{proof}

We are now ready to prove the upper bound.
\begin{proof}[Proof of the inequality
`$\leq$' in \eqref{2.33}.]
Fix $\ell>0$. By passing to a subsequence (which is still denoted by $n$) we assume without loss of generality assume that
\begin{align}\label{-infty}
\lim_{n\rightarrow\infty}
\mathbb E_{\Q^n}\left[f^n(X)-\frac{1}{2n\ell}\int_0^T|\psi^{\Q^n}_t|^2 dt\right]>-\iy.
\end{align}
Otherwise the statement is obvious.
Fix $n$ and introduce the $\Q^n$-martingale
$M_k:=X_{kT/n}$, $k=0,1,...,n$
where, recall
that
\begin{align}\notag
X_t=X_0+\sigma W^{\Q^n}_t+\sigma\int_{0}^t \psi^{\Q^n}_s ds+\mu t,\qquad t\in[0,T],
\end{align}
where $W^{\Q^n}$ is a Wiener process under $\Q^n$.
Observe that
for any $0\le k\le n-1$, we have
$$\mathbb E_{\Q^n}\left[\int_{kT/n}^{(k+1)T/n}\left(\psi^{\Q^n}_t+\mu/\sigma\right) dt\;\Big|\;\mathcal F_{kT/n}\right]=0.$$
This together with
Lemma \ref{lem:main1} and the scaling property of Brownian motion
gives
\begin{align*}\notag
&\E_{\Q^n}\left[\int_0^T|\psi^{\Q^n}_t|^2dt\right]\\
&\quad=
\E_{\Q^n}\left[\sum_{k=0}^{n-1}\mathbb E_{\Q^n}\left[\int_{kT/n}^{(k+1)T/n}|\psi^{\Q^n}_t|^2 dt\;\Big|\;\mathcal F_{kT/n}\right]\right]\\
&\quad=\E_{\Q^n}
\left[\sum_{k=0}^{n-1}\mathbb E_{\Q^n}\left[\int_{kT/n}^{(k+1)T/n}\left(\psi^{\Q^n}_t+\mu/\sigma\right)^2 dt\;\Big|\;\mathcal F_{kT/n}\right]\right]-\mu^2 T/\sigma^2\nonumber\\
&\quad\geq\E_{\Q^n}
\left[\sum_{k=0}^{n-1}\pen\left(\frac{n}{\sigma^2 T}\E_{\Q^n}\left[|M_{k+1}-M_k|^2\;\Big|\;\mathcal F_{kT/n}\right]\right)\right]-\mu^2 T/\sigma^2\\
&\quad\geq\E_{\Q^n}
\left[\sum_{k=0}^{n-1}\pen\left(\frac{n}{\sigma^2 T}
\E_{\Q^n}\left[|M_{k+1}-M_k|^2\;\Big|\;M_0,...,M_k\right]\right)\right]-
\mu^2 T/\sigma^2
\end{align*}
where the last inequality follows from the Jensen inequality for the convex function
$\pen(\cdot)$,

Finally, from the assumption %that $\lim_{n\rightarrow\infty}\pi(n,n\ell)>-\infty$and (\ref{new})}
\eqref{-infty},
it follows that we can apply
Lemma \ref{lem:main2} (i.e. (\ref{-K}) holds true for some constant $K$). We conclude
\begin{align*}\notag
&\mathbb E_{\Q^n}\left[f^n(X)-\frac{1}{2n\ell}\int_0^T|\psi^{\Q^n}_t|^2dt\right]\\
&\quad\leq C (h(n)n)^{1/8}+\mu^2 T/(2\ell\sigma^2 n)+\sup_{\nu\in\calV}
\E_{\mathbb P}\Big[f^n\Big(X^{(\nu)}\Big)-{\frac{1}{2\ell T}}\int_0^T\pen\Big(\frac{\nu_t}{{\sigma^2}}\Big)dt\Big].
\end{align*}
The proof is completed by (\ref{2.scale2})
and taking $n\rightarrow\infty$.\qed
\end{proof}

\subsection{Lower bound}\label{sec:42}

This section is devoted to the proof of the inequality
`$\geq$' in \eqref{2.33}.

\begin{proof}[Proof of the inequality
`$\geq$' in \eqref{2.33}.]

Fix $\ell>0$.
The proof is done in three steps. In the first step we construct a sequence of controls
on the Brownian probability space which asymptotically achieves the supremum on the right-hand side of
(\ref{2.33}) and have a simple structure. In the second step we
apply the processes $\theta^{\beta},\vartheta^{\beta}$ from Lemma \ref{lem:main1} in order to construct a sequence
of probability measures
$\Q^n\in\calQ^n$ together with their Girsanov's kernels $\psi^{\Q^n}$. The difficulty in this step stems from the fact that the processes $\theta^\beta$ and $\vartheta^\beta$ are constructed via the process $\bar W$, which in our case translates to $W^{\Q^n}$. Note that the measure $\Q^n$ is detemined by the Girsanov's kernel $\psi^{\Q^n}$. To overcome this technical difficulty, we use integration by parts and introduce a path-dependent SDE. As a by product, our process $\psi^{\Q^n}$ is measurable with respect to the original filtration $\calF^W$.
Finally, we show convergence of the payoff components.

{\it Step 1:}
For any $K>0$ and $n\in\N$, let $\calV^n_K\subset\calV$ be the set of all volatility processes
of the form
\begin{align}\label{form}
\nu_t=\sum_{k=0}^{n-1} \phi_k\left(W_0,W_{T/n},...,W_{kT/n}\right)\one_{t\in [kT/n, (k+1)T/n)}
\end{align}
where $\phi_k:\mathbb R^{k+1}\rightarrow [1/K,K]$, $k=0,1,....,n-1$, are continuous functions.

Set $\epsilon>0$. In this step we argue that
there exist $K=K(\epsilon)$ and $N=N(\epsilon)$ such that for any
$n>N$
\begin{align}\notag
&\sup_{\nu\in\calV}\E_{\mathbb P}\left[f\left(X^{(\nu)}\right)- \frac{1}{2\ell{T}}\int_{0}^T \pen\left(\frac{\nu_t}{\sigma^2}\right)dt \right]\\\label{lowerbound1}
&\quad<\epsilon+\sup_{\nu\in\calV^n_K}\E_{\mathbb P}\left[f^n\left(X^{(\nu)}\right)- \frac{1}{2\ell{T}}\int_{0}^T \pen\left(\frac{\nu_t}{\sigma^2}\right)dt \right].
\end{align}
To this end, observe first that, by standard density arguments, we will get the same supremum on the
right-hand side of (\ref{2.33})
 if instead of letting $\nu$ vary over all of $\calV$ there, we confine it to be of the form
\begin{align*}
\nu_t:= \sum_{j=0}^{J-1} \phi_j\left(W_{t_0},\dots,W_{t_{j}}\right) \one_{t\in [t_{j},t_{j+1})},  \quad t\in [0,T],
  \end{align*}
where $0=t_0<t_1<\dots<t_J=T$ is a finite deterministic partition of $[0,T]$
and each $\phi_j:\mathbb R^{j+1}\rightarrow\mathbb R_{+}$, $j=0,...,J-1$, is continuous,
bounded and bounded away from zero.

Let $\nu$ be of the above form. There exists $K$ such that $\nu\in [1/K,K]$ a.s. For any $n\in\mathbb N$
set
$$t^n_j:=\min\left\{t\in  \{0,T/n,2 T/n,...,T\}: t\geq t_j\right\}, \ \ j=0,1,...,J$$
and define $\nu^n\in\calV^n_K$ by
\begin{align*}\notag
\nu^n_t:= \sum_{j=0}^{J-1} \phi_j\left(W_{t^n_0},\dots,W_{t^n_{j}}\right) \one_{t\in [t^n_{j},t^n_{j+1})},  \quad t\in [0,T].
  \end{align*}
 Observe that
$\nu_n\rightarrow\nu$ in $L^2(dt\otimes\mathbb P)$. This together with the Lipschitz continuity of $f$ and
the Lipschitz continuity of $\pen$ on the interval $[1/(K\sigma^2),K/\sigma^2]$ gives that
\begin{align*}\notag
&\E_{\mathbb P}\left[f\left(X^{(\nu)}\right)- \frac{1}{2\ell{T}}\int_{0}^T \pen\left(\frac{\nu_t}{\sigma^2}\right)dt \right]\\
&\quad=\lim_{n\rightarrow\infty}\E_{\mathbb P}\left[f\left(X^{(\nu^n)}\right)- \frac{1}{2\ell{T}}\int_{0}^T \pen\left(\frac{\nu^n_t}{\sigma^2}\right)dt \right].
\end{align*}
 Thus, in order to establish (\ref{lowerbound1}) it remains to show that
 $$\lim_{n\rightarrow\infty}\E_{\mathbb P}\left[f^n\left(X^{(\nu^n)}\right)-f\left(X^{(\nu^n)}\right)\right]=0.$$
Indeed, from the Lipschitz continuity of $f$, the Burkholder--Davis--Gundy inequality and the fact that
$\nu^n\leq K$ for all $n$,
 it follows that there exists constants $\tilde C_1,\tilde C_2$ (independent of $n$) such that
 \begin{align*}\notag
&\E_{\mathbb P}\left[\left(f^n\left(X^{(\nu^n)}\right)-f\left(X^{(\nu^n)}\right)\right)^4\right]\\
&\quad\leq
\tilde C_1\E_{\mathbb P}\left[\max_{0\leq k\leq n-1}\sup_{kT/n\leq t\leq (k+1)T/n}\left(X^{(\nu^n)}_t-
X^{(\nu^n)}_{kT/n}\right)^4\right]\\
&\quad\leq \tilde C_1\sum_{k=0}^{n-1}\E_{\mathbb P}\left[\sup_{kT/n\leq t\leq (k+1)T/n}\left(X^{(\nu^n)}_t-
X^{(\nu^n)}_{kT/n}\right)^4\right]\\
&\quad\leq \tilde C_1 n \tilde C_2 n^{-2}\\
&\quad=\tilde C_1\tilde C_2/n.
\end{align*}
This completes the proof of (\ref{lowerbound1}).

{\it Step 2:}
Fix $K,A>0$ and choose $n\in\mathbb N$. Following (\ref{newnewnew})--(\ref{newnew}) we define the functions
$\beta:[1/K,K]\rightarrow \mathbb R_{+}$ and $\theta:[1/K,K]\times [0,T/n]\rightarrow\mathbb R$ by
\begin{align}\notag
\beta(u)&:=\frac{u/\sigma^2}{|u/\sigma^2-1|}\one_{\{u\neq\sigma^2\}},\\\notag
\theta(u,t)&:=\Big(\beta(u)T/n-\Big(T/n-t\Big)\Big)^{-1}\one_{\{u>\sigma^2\}}\\\label{def1}
&\quad-\Big(\beta(u)T/n+\Big(T/n-t\Big)\Big)^{-1}\one_{\{u<\sigma^2\}}.
\end{align}
Introduce the map
$\Phi^A:[1/K,K]\times \mathcal C[0,T/n]\rightarrow \mathcal C[0,T/n]$,
such that for any $u\in [1/K,K]$, $z\in \mathcal C[0,T/n]$, and $t\in [0,T/n]$,
\begin{align}\label{def}
\Phi^A_t(u,z):=(-A)\vee\left(\theta(u,t)z_t-\theta(u,0)z_0-\int_{0}^{t}z_s \frac{\partial\theta(u,s)}{\partial s}ds\right)\wedge A.
\end{align}
Observe that $\Phi^A_t(u,z)$ depends only on $z_{[0,t]}$.
For a given $u\in [1/K,K]$ and $y\in\mathbb R$ consider the path-dependent SDE
\begin{align}\label{SDENEW}
dY_t=dW_t-(\Phi^A_t(u,Y)-\mu/\sigma)dt, \ \ \ t\in [0,T/n], \ \ Y_0=y.
\end{align}
Let us notice that (for a given $u$) $\Phi^A(u,\cdot):\mathcal C[0,T/n]\rightarrow \mathcal C[0,T/n]$
is a Lipschitz continuous function with respect to the sup-norm. Thus,
Theorem 2.1 from Chapter 9 IX in \cite{RY} yields a unique strong solution for the above SDE. Obviously,
(\ref{SDENEW}) can be reformulated by
\begin{align*}\notag
d(U_t,Y_t)=(0,dW_t)-(0,\Phi^A_t(U_t,Y)-\mu/\sigma)dt, \ \ t\in [0,T/n], \  U_0=u, \  Y_0=y.
\end{align*}
The last formulation allows us to apply
Theorem 1 in \cite{K:1996} and we obtain the existence of a \textit{jointly} measurable function
$$\Psi^A:[1/K,K]\times \mathbb R\times\mathcal C[0,T/n]\rightarrow \mathcal C[0,T/n]$$
such that for any $u>0$ and $y\in\mathbb R$,
$Y_{[0,T/n]}:=\Psi^A(u,y,W_{[0,T/n]})$
is the unique strong solution to (\ref{SDENEW}).

Next, let $\nu\in\calV^n_K$ be given by (\ref{form}).
Define
inductively the random variables $u_k^{A,n,\nu}=u_k$ and
$Y^{A,n,\nu}_{[kT/n,(k+1)T/n]}=Y^{A}_{[kT/n,(k+1)T/n]}$, $k=0,1,...,n-1$, as follows.
Set $u_0:=\nu_0$, $Y^{A}_{[0,T/n]}:=\Psi^A(u_0,0,W_{[0,T/n]})$,
and  for $k=1,...n-1$,
\begin{align*}\notag
u_k&:=\phi_k\left(Y^{A}_0,...,Y^{A}_{kT/n}\right),\\
Y^{A}_{[kT/n,(k+1)T/n]}&:=S^k\left(\Psi^A\left(u_k,Y^{A}_{kT/n},\{W_{t+kT/n}-W_{kT/n}\}_{t\in [0,T/n]}\right)\right),
\end{align*}
  where $S_k:\mathcal C[0,T/n]\rightarrow \mathcal C[kT/n, (k+1)T/n]$ is the shift operator (bijection) given by
  $(S_k(z))_t:=z_{t-kT/n}$.

Recall the first paragraph of Section \ref{sec:42}. We now use the process $Y^A$ in order to generate {\it at the same time} a measure $\Q^A$ and its Girsanov's kernel $\psi^{\Q^A}$.
Observe that $(Y^{A}_{t})_{t\in [0,T]}$ satisfies the equation
\begin{align}\label{important}
Y^{A}_t=W_t+\frac{\mu t}{\sigma}-\sum_{k=0}^{n-1}\int_{t\wedge (kT/n)}^{t\wedge ((k+1)T/n)}\Phi^A_t\left(u_k,S^{-1}_k\left(Y^{A}_{[kT/n,(k+1)T/n]}\right)\right)dt.
\end{align}
Since $\Phi^A$ is a bounded function, then from the Girsanov's theorem we obtain that there exists
a probability measure $\mathbb Q^{A,n,\nu}=\mathbb Q^A\sim\mathbb P$ with
with finite entropy
$\E_{\mathbb Q^A}[\log(d\mathbb Q^A/d\mathbb P)]<\infty$
such that
$W^{\mathbb Q^A}:=Y^{A}$ is a Wiener process under $\mathbb Q^A$.

From (\ref{def}), (\ref{important}) and the integration by parts formula it follows that
\begin{align}\notag
&\psi^{\mathbb Q^A}_t=
(-A)\vee \left(\int_{kT/n}^t
\theta\left(\phi_k(W^{\mathbb Q^A}_0,...,W^{\mathbb Q^A}_{kT/n}),s-kT/n\right)dW^{\mathbb Q^A}_s\right)\wedge A-\frac{\mu}{\sigma}\\\label{drift}
&\mbox{for}\ \ t\in [kT/n,(k+1)T/n), \ \ k=0,1,..,n-1.
\end{align}
We end this step with arguing that there exists $N=N(K)$ such that for any $n>N(K)$ we have
$\mathbb Q^A\in\mathcal Q^n$.
First, we establish the martingale property.
Indeed, from (\ref{drift}) it follows that for any $0\le k\le n-1$, and
$t\in [kT/n,(k+1)T/n)$, the conditional distribution (under $\mathbb Q^A$) of $\psi^{\mathbb Q^A}_t$,
given $\mathcal F^W_{kT/n}$, is symmetric around $-\mu/\sigma$, and so
$(X_{kT/n})_{0\leq k\leq n}$ is $\mathbb Q^A$-martingale.

Finally, we establish the super-martingale property. Clearly, there exists a constant $\bar{c}>0$
which depends on $K$ such that $|\theta(u,t)|\leq \bar{c}$ for all
$u\in [1/K,K]$ and $t\in [0,T/n]$. This together with (\ref{2.scale1}) and (\ref{drift}) gives that there exists a parameter $N=N(K)$
such that for any $n>N(K)$
$$\mathbb E_{\mathbb Q^A}\left[|X_{(k+1)T/n}-X_{kT/n}|^3\;\Big|\;\mathcal F^W_{kT/n}\right]\leq h(n), \quad k=0,1,...,n-1,$$
as required.

{\it Step 3:} In this step we fix arbitrary $n>N(K)$ and $\nu\in\calV^K_n$. Then in view of
(\ref{lowerbound1}), in order to complete the proof it remains to show that
\begin{align}\notag
&
 \sup_{\mathbb Q\in\calQ^n}\mathbb E_{\mathbb Q}\left[f^n(X)-\frac{1}{2\lambda}\int_0^T|\psi^{\mathbb Q}_s|^2ds\right]\\\label{final}
 &\quad\ge
 \E_{\mathbb P}\left[f^n\left(X^{(\nu)}\right)- \frac{1}{2\ell{T}}\int_{0}^T \pen\left(\frac{\nu_t}{\sigma^2}\right)dt \right]-
\frac{\mu^2 T}{2n\ell\sigma^2}.
\end{align}
%\blue{From (\ref{2.2}) (take $(\gamma,\delta)\equiv 0$ in the denominator there), (\ref{dual}), and (\ref{WQ}), we obtain that}
By definition,
\begin{align}
\notag
&
 \sup_{\mathbb Q\in\calQ^n}\mathbb E_{\mathbb Q}\left[f^n(X)-\frac{1}{2\lambda}\int_0^T|\psi^{\mathbb Q}_s|^2ds\right]\\\label{final1}
 &\quad\ge
 \liminf_{A\rightarrow\infty}\;\mathbb E_{\mathbb Q^A}\left[f^n(X)-\frac{1}{2n\ell}\int_0^T\big|\psi^{\mathbb Q^A}_t\big|^2 dt\right].
\end{align}
From (\ref{def1}), (\ref{drift}), and Lemma \ref{lem:main1}
it follows that for any $A>0$ and $0\le k\le n-1$,
\begin{align*}\notag
&\mathbb E_{\mathbb Q^A}\left[\int_{kT/n}^{(k+1)T/n}\big|\psi^{\mathbb Q^A}_t\big|^2 dt\;\Big|\;\mathcal F^W_{kT/n}\right]\leq\frac{\mu^2 T}{\sigma^2 n}+\pen\left(\frac{\phi_k\left(W^{\mathbb Q^A}_0,...,W^{\mathbb Q^A}_{kT/n}\right)}{\sigma^2}\right).
\end{align*}
Thus,
\begin{align}\label{final2}
\mathbb E_{\mathbb Q^A}\left[\int_{0}^T \big|\psi^{\mathbb Q^A}_t\big|^2 dt\right]\leq
\mu^2 T/\sigma^2+\frac{n}{T}\mathbb E_{\mathbb P}\left[\int_{0}^T \pen\left(\frac{\nu_t}{\sigma^2}\right)dt\right].
 \end{align}
Finally, from (\ref{2.bac}) and
(\ref{drift}) we obtain
\begin{align}\label{referee}
\Q^A\circ \left(X_0,X_{T/n},...,X_T\right)^{-1}\To \mathbb P\circ \left(Y_0,Y_1,...,Y_n\right)^{-1} \quad \mbox{as} \quad A\rightarrow\infty
\end{align}
where
$Y_0:=X_0$ and for any $k=0,1,...,n-1$
\begin{align*}\notag
&Y_{k+1}-Y_k:=\sigma\left(W_{(k+1)T/n}-W_{kT/n}\right)\\
&+
\sigma\int_{kT/n}^{(k+1)T/n}\left(\int_{kT/n}^{t}
\theta\left(\phi_k(W_0,...,W_{kT/n}),s-kT/n\right)dW_s\right)dt\\
&=\sigma\int_{kT/n}^{(k+1)T/n}\left(1+\left((k+1)T/n-s\right)\theta\left(\phi_k(W_0,...,W_{kT/n}),s-kT/n\right)\right)dW_s
\end{align*}
where the last equality follows from
the Fubini theorem.
From the It\^{o} Isometry and (\ref{def1}) it follows that for any $k$
$$\mathbb E_{\mathbb P}\left[(Y_{k+1}-Y_k)^2|Y_0,...,Y_k\right]
=\frac{T}{n}\phi_k(W_0,...,W_{kT/n}).$$
We conclude
that
$$\mathbb P\circ \left(Y_0,Y_1,...,Y_n\right)^{-1}=\mathbb P\circ \left(X^{\nu}_0,X^{\nu}_\frac{T}{n},...,X^{\nu}_T\right)^{-1},$$
and so, from (\ref{referee}) and the Fatou's Lemma
$$
\liminf_{A\rightarrow\infty}\;\mathbb E_{\mathbb Q^A}\left[f^n(X)\right]\geq \E_{\mathbb P}\left[f^n\left(X^{(\nu)}\right)\right].
$$
This together with
(\ref{final1})--(\ref{final2}) completes the proof
of (\ref{final}).\qed
\end{proof}

$\\$\noindent
{\bf Acknowledgements.}
The authors thank the anonymous AE and reviewer for their valuable reports and comments which helped to improve the quality of this paper.

\footnotesize
% BibTeX users please use one of
%\bibliographystyle{abbrv}
%\bibliographystyle{spmpsci}      % mathematics and physical sciences
\bibliographystyle{abbrv}       % APS-like style for physics
\bibliography{refs}
\end{document}